 \documentclass[OJMO,NoBabel,manuscript]{cedram}

\usepackage[nameinlink,capitalise]{cleveref}
\usepackage{frame}
\usepackage{framed}
\usepackage{nccmath}
\crefalias{algorithm}{section}
\crefname{algorithm}{algorithm}{algorithms}
\Crefname{algorithm}{Algorithm}{Algorithms}
\creflabelformat{algorithm}{#2\textup{#1}#3}
\crefname{theo}{Theorem}{Theorems}

\DeclareMathOperator{\ri}{ri}

\DeclareMathOperator*{\argmin}{argmin}

\DeclareMathOperator{\FmuL}{\mathcal{F}_{\mu,L}}
\DeclareMathOperator{\Fmu}{\mathcal{F}_{\mu,\infty}}
\DeclareMathOperator{\Fccp}{\mathcal{F}_{0,\infty}}
\DeclareMathOperator{\PDg}{\mathrm{PD}}

\newcommand{\modif}[1]{{#1}}

\newcommand{\Rd}{\mathbb{R}^d}

\newcommand{\N}{\mathbb{N}}
\newcommand{\M}{\mathcal{M}}

\newcommand{\normsq}[1]{{\lVert #1\rVert ^2}}

\newcommand{\prox}{\mathrm{prox}}
\newcommand{\dom}{\mathrm{dom}}
\newcommand{\norm}[1]{{\lVert #1\rVert}}

\usepackage {tikz}
\usepackage{pgfkeys}
\usetikzlibrary{shapes,arrows}
\usepackage{pgfplots}
\pgfplotsset{compat=1.13}
\pgfplotsset{plotOptions/.style={%
		label style={font=\scriptsize},
		legend style={font=\scriptsize},
		tick label style={font=\scriptsize},
		solid,
		very thick
	}}
\definecolor{colorP1}{RGB}{55,126,184}  % blue
\definecolor{colorP2}{RGB}{228,26,28}  % red
\definecolor{colorP3}{RGB}{152,78,163} % purple
\definecolor{colorP4}{RGB}{77,175,74}  % green
\definecolor{colorP5}{RGB}{250, 150, 10} % orange
\definecolor{colorP6}{cmyk}{0,0.5,1,0}

\newcommand{\citet}[1]{\cite{#1}}
\newcommand{\citep}[1]{\cite{#1}}

\title[Inexact accelerated forward-backward]{A note on approximate accelerated forward-backward methods with absolute and relative errors, and possibly strongly convex objectives}

\author{\firstname{Mathieu} \lastname{Barr\'e}}
\address{INRIA (Sierra project-team) -- D\'ept. d'informatique,
	Ecole normale sup\'erieure, CNRS, 
	PSL Research University, Paris, France.}
\email{mathieu.barre@inria.fr}
\author{\firstname{Adrien} \lastname{Taylor}}
\address{INRIA (Sierra project-team) -- D\'ept. d'informatique,
	Ecole normale sup\'erieure, CNRS, 
	PSL Research University, Paris, France.}
\email{adrien.taylor@inria.fr}
\author{\firstname{Francis} \lastname{Bach}}
\address{INRIA (Sierra project-team) -- D\'ept. d'informatique,
	Ecole normale sup\'erieure, CNRS, 
	PSL Research University, Paris, France.}
\email{francis.bach@inria.fr}
\thanks{MB acknowledges support from an AMX fellowship. The authors acknowledge support from the European Research Council (grant SEQUOIA 724063).This work was funded in part by the french government under management of Agence Nationale de la recherche as part of the ``Investissements d'avenir'' program, reference ANR-19-P3IA-0001 (PRAIRIE 3IA Institute).} 
\begin{abstract} 
In this short note, we provide a simple version of an accelerated forward-backward method (a.k.a. Nesterov's accelerated proximal gradient method) possibly relying on approximate proximal operators and allowing to exploit strong convexity of the objective function. The method supports both relative and absolute errors, and its behavior is illustrated on a set of standard numerical experiments.

Using the same developments, we further provide a version of the accelerated proximal hybrid extragradient method of~\cite{monteiro2013accelerated} possibly exploiting strong convexity of the objective function.
\end{abstract}

\begin{document}

% Use the \maketitle command after the abstract
\maketitle

\section{Introduction}\label{sec:intro}
In this work, we consider a standard composite convex minimization problem of the form
\begin{equation}\label{eq:opt}
    \min_{x\in\mathbb{R}^d} \left\{F(x)\equiv f(x)+g(x) \right\},
\end{equation}
where $f:\mathbb{R}^d\rightarrow \mathbb{R}$ is a $L$-smooth convex function (with $0<L<\infty$), and $g:\mathbb{R}^d\rightarrow \mathbb{R}\cup\{+\infty\}$ is a proper closed convex function. In addition, we allow either $f$ or $g$ to be possibly $\mu$-strongly convex. In this setting, we propose an inexact accelerated forward-backward method for solving~\eqref{eq:opt} relying on the access to the gradient of $f$, and to an iterative routine for approximating the proximal operator of $g$.

\paragraph*{Relation to previous works}  The main algorithms presented in this note were originally presented in~\citet{barre2020principled}, along with their worst-case analyses. It was removed from~\citet{barre2020principled} for exposition and length purposes. The same methods were then re-analyzed and used by~\citet{alves2021variants} for accelerating higher-order tensor algorithms.

When the proximal operator of $g$ is readily available, the method presented below becomes a variant of standard accelerated (or fast) forward-backward (or proximal gradient) methods for convex minimization, see e.g.,~\citet{nesterov2013gradient,beck2009fast}, and the introductory survey by~\citet{d2021acceleration}. 

Purely backward versions ($f=0$) emerged earlier from the works of~\citet{guler1992new} and~\citet{monteiro2013accelerated,salzo2012inexact}, whereas the first purely forward version ($g=0$) was developed by~\citet{Nesterov:1983wy}. The first inexact versions of accelerated forward-backward methods that we are aware of were presented in~\citet{schmidt2011convergence,villa2013accelerated,jiang2012inexact}, whereas versions with relative errors appeared more recently in \citet{millan2019inexact,Bello2020}. In contrast, our method allows handling different types of error (namely absolute and relative errors of different types), while allowing to exploit strong convexity of $f$ or $g$---see e.g.~\citet{Book:Nesterov,nesterov2013gradient,chambolle2016introduction}, for original analyses in the strongly convex case, when the proximal operator of $g$ is readily available. The same developments allow obtaining a version of the accelerated hybrid proximal extragradient method (A-HPE)---in the spirit of~\citet{monteiro2013accelerated}---for exploiting strong convexity of the problem at hand.

The notion of an ``approximate proximal point'' used in this note (see \Cref{s:prox}) was used in a few previous works, starting with the hybrid extragradient method \citep{solodov1999hybrid,solodov2000comparison}. It was also used for its accelerated version~\citep{monteiro2013accelerated} and in the context of another forward-backward splitting method~\citep{millan2019inexact}. In these works, the primal-dual requirement is presented under a different formulation involving the notion of $\varepsilon$-subdifferentials \cite[Section 3]{brondsted1965subdifferentiability} (or $\varepsilon$-enlargement in the context of monotone operators \citep{burachik1997enlargement,burachik1998varepsilon,burachik2001robustness}). Among others, a variant of the hybrid extragradient method was also studied in~\citep{burachik2001robustness} under both absolute and relative errors, similar in spirit with the accelerated methods presented below. A survey on common notions of ``approximate proximal point'' used in the literature can be found in \cite[Section 2]{barre2020principled_new}.

\paragraph*{Paper organization and contribution} This note is organized as follows. First, we give some basic results and notations in \cref{s:backg}. We provide the inexact accelerated forward-backward in \cref{s:accel_FB}, along with a worst-case analysis, relying on a standard Lyapunov argument (for which we provide symbolic notebooks, helping the reader reproducing the algrebraic part of the proof without pain). Numerical experiments illustrating the practical behavior of the method are then provided in \cref{s:numerics}. After that, \cref{s:ahpe} shows how to slightly modify the proof for obtaining an accelerated hybrid proximal extragradient method~\citep{monteiro2013accelerated}, specifically for the case $f=0$. We draw some conclusions in \cref{s:ccl}.

\paragraph*{Notations} We refer to classical textbooks~\citep{Book:Rockafellar,hiriart2013convex} for standard elements of convex analysis. We use the notation $\Fccp(\Rd)$ to denote the set of closed convex proper functions on $\Rd$. The corresponding subset of closed convex proper functions that are $\mu$-strongly convex and $L$-smooth (with $0\leq \mu < L\leq \infty$) is denoted $\FmuL(\Rd)$. That is, $h\in\FmuL(\mathbb{R}^d)$ if and only if
\begin{itemize}
    \item ($\mu$-strong convexity) $\forall x,y\in\mathbb{R}^d$, $s_h(x)\in\partial h(x)$, $s_h(y)\in\partial h(y)$, it holds $\|s_h(x)-s_h(y)\|\geq \mu \|x-y\|$,
    \item ($L$-smoothness) $\forall x,y\in\mathbb{R}^d$, $s_h(x)\in\partial h(x)$, $s_h(y)\in\partial h(y)$, it holds $\|s_h(x)-s_h(y)\|\leq L \|x-y\|$,
\end{itemize}
where $\partial h(x)$ denotes the subdifferential of $h$ at $x\in\mathbb{R}^d$. When $h\in\FmuL(\mathbb{R}^d)$ with $L<\infty$, we use $h'(x)$ to denote the unique element $h'(x)\in\partial h(x)$ (i.e. the gradient of $h$ at $x$).

\paragraph*{Codes} For helping the reader reproducing the analytical results (via Mathematica notebooks) as well as numerical experiments, our code is available at
\begin{center}
    \url{https://github.com/mathbarre/StronglyConvexForwardBackward}.
\end{center}

\section{Background results}\label{s:backg}
\subsection{ Smooth strongly convex functions}\label{s:ssc_functions}
We recall some standard inequalities satisfied by smooth convex and strongly convex functions, which we use in the sequel for exploiting strong convexity and smoothness, see e.g.~\citet{Book:Nesterov}. 
\begin{prop}[$\mu$-strong convexity]\label{prop:strcvx}
Let $g\in\Fmu(\Rd)$. For all $x,y\in\Rd$ and all $s_g(x)\in \partial g(x)$ it holds that
\[g(y)\geq g(x) + \langle s_g(x),y-x\rangle + \tfrac{\mu}{2}\normsq{x-y}. \]
\end{prop}
%\begin{proof}
%    Use subgradient definition (see e.g., \cite[Definition 3.1.6]{Book:Nesterov}) on $g-\tfrac{\mu}{2}\normsq{\cdot}\in\Fccp(\Rd)$.
%\end{proof}
\begin{prop}[$L$-smoothness \& convexity]\label{prop:smooth}
Let $f\in\mathcal{F}_{0,L}(\Rd)$ with $L <+\infty$. For all $x,y\in\Rd$ it holds that
\[f(y)\geq f(x) + \langle f'(x), y-x\rangle + \tfrac{1}{2L}\normsq{f'(x)-f'(y)}. \]
\end{prop}
%\begin{proof}
%see e.g., \cite[Theorem 2.1.5]{Book:Nesterov}.
%\end{proof}
%In the sequel, the use of the inequalities provided by~\Cref{prop:strcvx} and~\Cref{prop:smooth} are motivated by their \emph{interpolation} (or extension) properties, meaning that worst-case analyses of fixed-step first-order methods in the setup of~\eqref{eq:opt} can always be done using only those; see e.g.,~\cite[Theorem 4]{taylor2017smooth}.

%In the sequel, the use of the inequalities provided by~\Cref{prop:strcvx} and~\Cref{prop:smooth} are motivated by their \emph{interpolation} (or extension) properties. It can be shown that principled worst-case analyses of fixed-step first-order methods in the setup of~\eqref{eq:opt} can always be performed using only those particular inequalities; see e.g.,~\citep{taylor2017smooth}.

In the sequel, the use of the inequalities provided by~\Cref{prop:strcvx} and~\Cref{prop:smooth} are motivated by their \emph{interpolation} (or extension) properties; that is, the analyses provided below were obtained following a principled approach to worst-case analyses of first-order methods, see e.g.,~\citet{taylor2017smooth} or~\citet{barre2020principled_new} specifically for the cases of methods relying on approximate proximal operations.

\subsection{ Proximal operations}\label{s:prox}
The proximal operation is a basic primitive that is widely used in modern optimization methods; it is a central building blocks in many optimization algorithms, see e.g.,~\citet{parikh2014proximal,ryu2016primer}. The proximal operator of a function $g\in \Fccp(\Rd)$ with step size $\lambda > 0$ is defined as 
\begin{equation}\label{eq:prox}
    \prox_{\lambda g}(z) = \underset{x\in\Rd}{\argmin}\,\left\{ \lambda g(x) + \tfrac{1}{2}\normsq{x-z}\right\},
\end{equation}
with $z\in \Rd$. When $g\in \Fccp(\Rd)$, the proximal operation is well defined, and its solution is unique. A comprehensive list of cases where~\eqref{eq:prox} has an analytical solution is provided in~\citet{chierchia2020proximity}. In other cases, the proximal operator has to be approximated. For doing that, one can define the following primal and dual problems associated to the proximal operation 
\begin{align}
     \underset{x\in\Rd}{\min}&\,\{\Phi_\mathrm{p}(\modif{x;\,z})\equiv   \lambda g(x) + \tfrac{1}{2}\normsq{x-z}\}, \label{eq:P}\tag{P}\\
     \underset{v\in\Rd}{\max}&\,\{\Phi_\mathrm{d}(\modif{v;\,z})\equiv  -\lambda g^*(v) - \tfrac{1}{2}\normsq{z - \lambda v} + \tfrac{1}{2}\normsq{z}\},\label{eq:D} \tag{D}
\end{align}
where $g^*\in\Fccp(\Rd)$ is the Fenchel conjugate of $g$. Let us further note that $\prox_{\lambda g}(z)$ is the unique solution to~\eqref{eq:P}, and that $\prox_{g^*/\lambda}(z/\lambda)$ is the unique solution of \eqref{eq:D}. In this context, the primal and dual solutions are linked by the well-known Moreau's identity $\prox_{\lambda g}(z) +\lambda\,\prox_{{g^*}/{\lambda}}\left(z/{\lambda}\right)= z$.

Under relatively weak conditions (such as $\ri(\dom g) \neq \emptyset$, see e.g.,~\cite[Corollary 31.2.1]{Book:Rockafellar}), strong duality holds between~\eqref{eq:P} and~\eqref{eq:D} and hence
\[\underset{x\in\Rd}{\min}\,\Phi_\mathrm{p}(\modif{x;\,z}) = \underset{v\in\Rd}{\max}\,\Phi_\mathrm{d}(\modif{v;\,z}). \]
Motivated by those elements, we use the quantity 
\begin{equation}\label{eq:PD}\tag{PD}
    \PDg_{\lambda g}(x,v;\,z) = \Phi_\mathrm{p}(\modif{x;\,z}) - \Phi_\mathrm{d}(\modif{v;\,z}), 
\end{equation}
for quantifying how well $(x,v)$ approximates the pair $(\prox_{\lambda g}(z),\prox_{g^*/\lambda}(z/\lambda))$, in the sequel. 

\subsection{ A notion of approximate proximal point}% for strongly convex objectives}
In this section, we define the notion of approximate proximal point of $g\in\Fmu(\Rd)$ used throughout the paper (see \Cref{sec:intro}~\S``{Relation to previous works}'' for historical references for the case $\mu=0$). This notion features two parameters: a tolerance and a lower bound on the strong convexity parameter of $g$ (possibly $0$). The estimate of the strong convexity is used for relating proximal points of $g(\cdot)$ in terms of that of $g_\mu(\cdot)=g(\cdot)-\tfrac{\mu}{2}\|\cdot\|^2\in\Fccp(\Rd)$, and the tolerance is used for quantifying the quality of an approximate solution to the proximal problem on $g_\mu(\cdot)$, which simplifies the analyses below. More precisely, for $g\in\Fmu(\Rd)$, it is relatively straightforward to verify that
%When $g$ is $\mu$-strongly convex, we can exploit the structure of the objective function by using the strong convexity knowledge inside the primal dual inexactness criterion, in order to simplify the analysis. Indeed, for $g\in\Fmu(\Rd)$, we notice that 
\[\prox_{\lambda g}(z) = \prox_{\tfrac{\lambda}{1+\lambda\mu}g_\mu}\left( \tfrac{z}{1+\lambda\mu}\right),\]
with $g_\mu(x) = g(x) - \tfrac{\mu}{2}\normsq{x}$. This observation motivates the introduction of the following inexactness criterion.
\begin{defi}\label{def:criterion} Let $\mu> 0$, $g\in \Fmu(\Rd)$, and let $\lambda>0$ be a step size and $\varepsilon\geq0$ be a tolerance. For a triplet $(x,v,y)\in \Rd\times \Rd\times\Rd$ we use the notation 
\[(x,v) \approx_{\varepsilon,\mu} \left(\prox_{\lambda g}(y),\prox_{\tfrac{g*}{\lambda}}(\tfrac{y}{\lambda}) \right), \]
for denoting that
\[\PDg_{\tfrac{\lambda}{1+\lambda\mu} g_\mu }\left(x,v-\mu x;\tfrac{y}{1+\lambda\mu}\right) \leq \varepsilon,\]
with $g_\mu(x) = g(x)-\tfrac{\mu}{2}\|x\|^2$ and $\PDg$ is the primal-dual gap of the proximal problem defined in \eqref{eq:PD}.
\end{defi}
In the following technical lemma, we provide an explicit expression for quantifying the quality of a triplet $(x,v,y)\in \Rd\times \Rd\times\Rd$ in light of \Cref{def:criterion}.
\begin{lemm}
Let $\mu\geq 0$, $g\in \Fmu(\Rd)$, and let $\lambda>0$ be a step size and $(x,v,z)\in\Rd\times\Rd\times\Rd$. The following equality holds
\begin{equation}\label{eq:pdgap-2}
\begin{aligned}
    \PDg_{\tfrac{\lambda}{1+\lambda\mu} g_\mu}(x,v-\mu x;\tfrac{z}{1+\lambda\mu})
    =& \;\tfrac{1}{2(1+\lambda\mu)^2}\normsq{x - z +\lambda v}\\
    &+ \tfrac{\lambda}{1+\lambda\mu}\left(g(x)-g(w)+\tfrac{\mu}{2}\normsq{x-w} - \langle x-w,v \rangle \right),
\end{aligned}
\end{equation}
with $g_\mu(\cdot) = g(\cdot)-\tfrac{\mu}{2}\|\cdot\|^2$ and $w \in \Rd$ satisfying $v-\mu x + \mu w \in \partial g (w)$ (i.e., $w \in \partial g_\mu^*(v-\mu x)$).
\end{lemm}
\begin{proof}
    \begin{equation}\label{eq:pdgap-1}
\begin{aligned}
    \PDg_{\tfrac{\lambda}{1+\lambda\mu} g_\mu}(x,v-\mu x;\tfrac{z}{1+\lambda\mu}) =& \;\tfrac{1}{2}\normsq{x-\tfrac{z}{1+\lambda \mu}+\tfrac{\lambda}{1+\lambda\mu}(v-\mu x)}\\
    &+ \tfrac{\lambda}{1+\lambda\mu}\left(g_\mu(x) + g_\mu^*(v-\mu x) - \langle x,v-\mu x \rangle \right)\\
    =& \;\tfrac{1}{2(1+\lambda\mu)^2}\normsq{x - z +\lambda v}\\
    &+ \tfrac{\lambda}{1+\lambda\mu}\left(g(x)+\tfrac{\mu}{2}\normsq{x} + g_\mu^*(v-\mu x) - \langle x,v \rangle \right).
\end{aligned}
\end{equation}
In particular 
\begin{equation*}
    g_\mu^*(v-\mu x) = \max_y\,\{ \langle y,v-\mu x \rangle - g(y)+\tfrac{\mu}{2}\normsq{y}\},
\end{equation*}
and by choosing $w \in \Rd$ such that $v-\mu x+\mu w \in \partial g(w)$ we get
\begin{equation}
    g_\mu^*(v-\mu x) = \langle w,v-\mu x \rangle - g(w)+\tfrac{\mu}{2}\normsq{w}.
\end{equation}
Finally, using the expression of $g_\mu^*(v-\mu x)$ in \eqref{eq:pdgap-1} leads to the desired results.
\end{proof}

In the next section, we present an inexact accelerated forward-backward method where inexactness in proximal computations are measured using the primal-dual criterion from \cref{def:criterion}.
		
\section{An inexact accelerated forward-backward method}\label{s:accel_FB}
In this section, we provide the main contribution of this work, namely \Cref{eq:algo2}. This method aims at solving problem~\eqref{eq:opt} when the gradient of $f$ is readily available and the proximal operator of $g$ can be efficiently approximated within a target precision (e.g., by an iterative method). It further allows to exploit $g$ to be $\mu$-strongly convex. In the case where $f$ is strongly convex, one can shift this strong convexity to $g$ instead (by removing the corresponding quadratic of $f$ and adding it to $g$). Of course, any under-approximation of $\mu$ can be used within the method.

The worst-case analysis is based on a simple Lyapunov (or potential) argument, following the now standard template for accelerated schemes as in~\citet{Nesterov:1983wy}, for which surveys are provided in e.g.,~\citet{bansal2019potential,wilson2021lyapunov}, and~\cite[Chapter 4]{d2021acceleration}. As a byproduct of the analysis, the method does not require an accurate estimate of the smoothness constant $L$, whose estimation is improved on the fly using  standard backtracking tricks, similar in spirit with~\citet{Nesterov:1983wy,beck2009fast}.

The algorithm below builds on approximations of the forward-backward operator (with step sizes $\lambda_k$) of problem~\eqref{eq:opt}. More precisely, it relies on primal-dual pairs $(x_{k+1},\,v_{k+1})$ approximating the forward-backward operator evaluated at some iterates $y_k$, and satisfying 
\[ (x_{k+1},\;v_{k+1})\approx_{\varepsilon_k,\mu}\left(\prox_{\lambda_{k} g}(y_{k}-\lambda_k f'(y_k)),\prox_{\tfrac{g^*}{\lambda_{k}}}(\tfrac{y_{k}-\lambda_k f'(y_k)}{\lambda_k})\right), \]
where $\varepsilon_k$ encodes some approximation level. In this work, this error term is parameterized by three sequences of nonnegative scalars $\{\sigma_k\}_k$, $\{\zeta_k\}_k$, $\{\xi_k\}_k$ that can be chosen by the user for possibly mixing both \emph{relative} (or multiplicative) and \emph{absolute} (or additive) error terms
\[ \varepsilon_k =\tfrac{\sigma_k^2}{2(1+\lambda_k \mu)^2}\|x_{k+1}-y_k\|^2+\tfrac{\zeta_k^2\lambda_k^2}{2(1+\lambda_k \mu)^2}\|v_{k+1}+f'(y_k)\|^2+\tfrac{\lambda_k\xi_k}{2(1+\lambda_k \mu)^2},\]
where $\{\xi_k\}_k$ parameterizes the absolute error term, and where $\{\sigma_k\}_k$ and $\{\zeta_k\}_k$ parametrize two types of relative errors. Of course, convergence properties of the algorithm depend on the choice of those sequences of parameters, as provided in \cref{cor:speed_strcvx} and \cref{cor:speed_cvx} below. Examples of simple rules for $\{\sigma_k\}_k$, $\{\zeta_k\}_k$, $\{\xi_k\}_k$ are provided in \cref{s:numerics} (typically, $\{\sigma_k\}_k$, $\{\zeta_k\}_k$ can be chosen constant, whereas $\{\xi_k\}_k$ should be either identically $0$ or decreasing fast enough).

Before going into the algorithm itself, let us mention that the backtracking line-search strategy \eqref{eq:backtracking} for estimating the smoothness constant builds on the condition
\begin{equation}\label{eq:smoothness}
    f(y_{k})  \geq f(x_{k+1}) + \langle f'(x_{k+1}),y_{k} - x_{k+1}\rangle + \tfrac{\lambda_k}{2(1-\sigma_k^2)}\| f'(y_{k}) -  f'(x_{k+1})\|^2, \tag{Smooth}
\end{equation}
where $x_k$'s and $y_k$'s are some iterates. In particular, picking $\lambda_k \in (0,\tfrac{1-\sigma_k^2}{L}]$ (hence depending on the true smoothness constant $L$) guarantees~\eqref{eq:smoothness} to be satisfied without backtracking, as, when $f \in \mathcal{F}_{0,L}(\Rd)$, \cref{prop:smooth} holds.

\begin{rema}[Related methods]
When the objective function is not strongly convex (i.e. $\mu=0$), the update rules of \Cref{eq:algo2} are very similar to those of the accelerated inexact forward-backward methods from \cite[Algorithm 3]{millan2019inexact} (when $\zeta_k=0$ and $\xi_k=0$) or \cite[Algorithm 2]{Bello2020} (when $\sigma_k=0$ and $\xi_k=0$). Compared to those works in this setup, \Cref{eq:algo2} allows using both relative and absolute errors while having a backtracking strategy. Note also the similarities with some inexact FISTA \citep{villa2013accelerated,schmidt2011convergence}, although these methods do not re-use explicitly the dual direction $v_{k+1}$ and focus on absolute error terms (i.e., $\sigma_k=\zeta_k=0$). Finally, when the computation of the proximal operator is exact, we recover one of the many variants of an accelerated forward-backward method; see for example~\citet{nesterov2013gradient,tseng2008accelerated,beck2009fast}; we refer to~\cite[Chapter 4]{d2021acceleration} and the references therein for further discussions.\end{rema}

\subsection{ Algorithm}\label{subsec:FB}
\begin{oframed}
	\textbf{An inexact accelerated forward-backward method } (\Cref{eq:algo2}) \label[algorithm]{eq:algo2}
	\begin{itemize}
		\item[] {\bf{}Input:}
		\begin{itemize}
		    \item Objective function: $f\in\mathcal{F}_{0,L}(\Rd)$, $g\in\mathcal{F}_{\mu,\infty}(\Rd)$, and $\mu\geq 0$.
		    \item Initial point: $x_0\in\mathbb{R}^d$.
		    \item Initial step size: $\lambda_0 >0$.
		    \item Tolerance parameters: sequences $\{\sigma_k\}_k,\,\{\zeta_k\}_k$ with $\sigma_k,\,\zeta_k\in[0,1)$, and $\{\xi_k\}_k$ with $\xi_k\geq 0$.
		    \item Backtracking parameters $0<\alpha <1$ and $\beta \geq 1$.
		\end{itemize}
		\item[] {\bf{}Initialization:} $z_0=x_0$, $A_0 = 0$. 
		\item[] {\bf{}Run:} For $k=0,1,\hdots$:
		\begin{align}
		    \eta_k &= (1-\zeta_k^2)\lambda_k\label{eq:init}\\
			A_{k+1} &= A_k + \tfrac{\eta_k + 2 A_{k}\mu\eta_k +\sqrt{ \eta_k^2 + 4\eta_k A_{k}(1+\eta_k\mu)(1+A_{k}\mu)}}{2} \nonumber\\
			y_{k} &= x_{k}+\tfrac{(A_{k+1}-A_{k})(A_{k}\mu+1)}{A_{k+1}+A_k(2A_{k+1}-A_k)\mu }(z_{k}-x_{k}) \nonumber\\
			\varepsilon_k &=\tfrac{\sigma_k^2}{2(1+\lambda_k \mu)^2}\|x_{k+1}-y_k\|^2+\tfrac{\zeta_k^2\lambda_k^2}{2(1+\lambda_k \mu)^2}\|v_{k+1}+f'(y_k)\|^2+\tfrac{\lambda_k\xi_k}{2(1+\lambda_k \mu)^2} \nonumber\\
			(x_{k+1},\;v_{k+1})&\approx_{\varepsilon_k,\mu}\left(\prox_{\lambda_{k} g}(y_{k}-\lambda_k f'(y_k)),\prox_{\tfrac{g^*}{\lambda_{k}}}(\tfrac{y_{k}-\lambda_k f'(y_k)}{\lambda_k})\right)\nonumber\\
		 \mathbf{\large[}\text{If } &\eqref{eq:smoothness} \text{ is not satisfied, set }
		 \lambda_k \leftarrow \alpha \lambda_k \tag{Btr}\label{eq:backtracking} \text{ and go back to step~\eqref{eq:init}\large]}\nonumber\\%\\
		 %\text{Go back} &  \text{ to step } \eqref{eq:init} \mathbf{\large]}\nonumber\\
			z_{k+1} &= z_{k}+\tfrac{A_{k+1}-A_{k}}{1+\mu A_{k+1}}\left(\mu(x_{k+1}-z_{k}) - (v_{k+1} + f'(y_{k}))\right)\nonumber\\
			\lambda_{k+1} &= \beta \lambda_{k} \nonumber
			\end{align}
			\item[] {\bf{}Output:} $x_{k+1}$
	\end{itemize}  
\end{oframed}

The following theorem contains the main (Lyapunov-based) ingredient of the worst-case analysis.
\begin{theo}\label{thm:deacres-potential} 
Let $f\in\mathcal{F}_{0,L}(\Rd)$, $g \in \Fmu(\Rd)$, $F\equiv f+g$, $k\geq 0$, parameters $\sigma_k,\zeta_k \in [0,1)$, $\xi_k \geq 0$ and some $\lambda_k>0$ such that \eqref{eq:smoothness} is satisfied. For any $x_k,z_k\in\Rd$, and $A_k\geq0$, %two consecutive iterates of the inexact accelerated forward-backward method satisfy
it holds that
\begin{equation}\label{eq:deacres-potential}
\begin{aligned}
A_{k+1}&(F(x_{k+1})-F(x_\star))+\tfrac{1+\mu A_{k+1}}{2}\normsq{z_{k+1}-x_\star}\leq A_k (F(x_k)-F(x_\star))+\tfrac{1+\mu A_k}{2}\normsq{z_k-x_\star}+ \tfrac{A_{k+1}}{2}\xi_k,
\end{aligned}
\end{equation}
with $x_\star\in\argmin_x F(x)$, and where $z_{k+1}$ and $x_{k+1}$ are constructed by one iteration of \Cref{eq:algo2}.
\end{theo}

The proof of this Theorem is deferred to \cref{subsec:proof-FB}. The following (classical) corollary establishes that the growth rate of the sequence $\{A_k\}_k$ drives the convergence rate of the worst-case guarantee.  Those factors $A_{k+1}$, controlling the convergence rate, were greedily chosen (as large as possible) while enforcing \eqref{eq:deacres-potential} to hold. 

% Before studying the growth rate the the $A_k$'s, we combine the results of \cref{thm:deacres-potential} for various values of $k$ to obtain the following worst-case bound after $N$ iterations of the inexact accelerated forward-backward method.
\begin{coro}\label{cor:FB}
Let $f\in\mathcal{F}_{0,L}(\Rd)$, $g \in \Fmu(\Rd)$ and $
F\equiv f+g$. Let $x_0\in\Rd$, $\lambda_0$ be a positive initial step size, $\alpha\in(0,1)$ and $\beta\geq 1$ be some backtracking parameters, sequences (relative error parameters) $\{\sigma_k\}_k$, $\{\zeta_k\}_k$, satisfying $\sigma_k,\zeta_k\in[0,1)$ and a sequence (absolute error parameters) $\{\xi_k\}_k$ with $\xi_k\geq 0$. Let $x_N\in\Rd$ be the output after $N\in\N^*$ iterations of \Cref{eq:algo2} on $F$ initiated at $x_0\in\Rd$, it holds that
\[ F(x_N) - F(x_\star) \leq \tfrac{1}{2A_N}\normsq{x_0-x_\star} + \sum_{i=0}^{N-1}\tfrac{A_{i+1}}{2A_N}\xi_i,\]
where $x_\star \in \argmin_x F(x)$.
\end{coro}
\begin{proof}
We denote by $\Phi_k$ the quantity (a.k.a., the Lyapunov/potential function)
\[\Phi_k = A_k(F(x_k)-F(x_\star)) + \tfrac{1+\mu A_k}{2}\normsq{z_k-x_\star}, \]
for $k \geq 0$. \cref{thm:deacres-potential} allows nesting the $\Phi_k$'s together as 
\[\Phi_N \leq \Phi_{N-1} + \tfrac{A_{N}}{2}\xi_{N-1} \leq \hdots \leq \Phi_1 + \sum_{i=1}^{N-1}\tfrac{A_{i+1}}{2}\xi_i\leq \Phi_0 +\sum_{i=0}^{N-1}\tfrac{A_{i+1}}{2}\xi_i.\]
We reach the target conclusion using $A_N(F(x_N)-F(x_\star)) \leq \Phi_N$, together with $z_0=x_0$ and $A_0 =0$.
\end{proof}

Let us note that when $\mu=0$, we recover a composite version of the A-HPE method~\citep{monteiro2013accelerated}. In that case, we can bound $A_k\geq\tfrac14\left(\sum_{i=0}^{k-1}\sqrt{\eta_k}\right)^2 \geq \tfrac{\eta_\mathrm{min}}{4}k^2$, assuming the existence of some $\eta_\mathrm{min}\leq \eta_k$ for all $k\geq 0$. Such a lower bound on $\eta_k$ exists as soon as the parameters $\{\sigma_k\}_k$, $\{\zeta_k\}_k$ are well chosen (see for example \cref{cor:speed_strcvx} and \cref{cor:speed_cvx} below), and due to the $L$-smoothness of the function. Similarly, when $\mu>0$, $A_k$'s are growing exponentially as
\begin{align*}
    A_{k+1} &= A_{k} + \tfrac{\eta_k+2A_k \eta_k \mu + \sqrt{4 \eta_k A_k (A_k \mu +1) (\eta_k \mu +1)+\eta_k^2}}{2}\\
    & \geq A_{k}(1 + \eta_k\mu) +A_k \sqrt{\eta_k\mu(1+\eta_k\mu)}\\
    & = A_{k}/\left(1-\sqrt{\tfrac{\eta_k  \mu}{1+\eta_k \mu}}\right),
\end{align*}
with $A_1 > 0$, reaching $1/A_k \leq \eta_{\mathrm{min}}\left(1-\sqrt{\tfrac{\eta_{\mathrm{min}}\mu}{1+\eta_{\mathrm{min}}\mu}}\right)^{k-1}$ assuming again the existence of some $\eta_\mathrm{min}\leq \eta_k$ for all $k\geq 0$. The following corollaries provide more precise convergence bounds for \Cref{eq:algo2}, by quantifying the growth rate of the $A_k$'s, for some particular choices of parameters $\{\sigma_k\}_k$~(constant), $\{\zeta_k\}_k$~(constant), and $\{\xi_k\}_k$ (parameterized function of $k$), linking the behavior of the decrease rate of the absolute errors $\xi_k$ with the convergence bound.
\begin{coro}\label{cor:speed_strcvx}
Let $f\in\mathcal{F}_{0,L}(\Rd)$, $g \in \Fmu(\Rd)$ and $F\equiv f+g$. Let $x_0\in\Rd$, $\lambda_0$ be an initial positive step size, $\alpha\in(0,1)$ and $\beta\geq 1$ be some backtracking parameters, sequences (relative error parameters) $\sigma_k = \sigma$, $\zeta_k =\zeta$ with $\sigma,\zeta\in[0,1)$ and a sequence (absolute error parameters) $\xi_k=C \rho^k$ with $C,\rho>0$. Let $x_N\in\Rd$ be the output after $N\in\N^*$ iterations of \Cref{eq:algo2} on $F$ initiated at $x_0\in\Rd$, it holds that
\[F(x_N) - F(x_\star)\leq \tfrac{1}{2\eta}\left(1-\sqrt{\tfrac{\eta\mu}{1+\eta\mu}}\right)^{N-1}\|x_0-x_\star\|^2+\left\{\begin{array}{ll}
   \medmath{\frac{C}{ 2(1-\sqrt{\tfrac{\eta\mu}{1+\eta\mu}}-\rho)}\left(1-\sqrt{\tfrac{\eta\mu}{1+\eta\mu}} \right)^{N}}\hspace{-.1cm}& \textrm{if } \medmath{\rho  < 1-\sqrt{\tfrac{\eta\mu}{1+\eta\mu}}}, \\
     \medmath{\frac12 C N \left(1-\sqrt{\tfrac{\eta\mu}{1+\eta\mu}} \right)^{N-1}} & \textrm{if } \medmath{\rho  = 1-\sqrt{\tfrac{\eta\mu}{1+\eta\mu}}}, \\
     \medmath{\frac{C}{2(\rho -  1+\sqrt{\tfrac{\eta\mu}{1+\eta\mu}})}\rho^N} &\textrm{if } \medmath{\rho  > 1-\sqrt{\tfrac{\eta\mu}{1+\eta\mu}}},\end{array} \right.\]
     for some $\eta = \underset{i=0,\hdots,N-1}{\min}\,\eta_i\geq \eta_{\mathrm{min}} = (1-\zeta^2)\min\left(\lambda_0, \tfrac{\alpha(1-\sigma^2)}{L}\right)$ and where $x_\star \in \argmin_x F(x)$.
\end{coro}
\begin{proof} Starting from the conclusion of \cref{cor:FB}, we obtain the desired result using classical properties of geometric sums along with $A_k \leq \left(1-\sqrt{\tfrac{\eta\mu}{1+\eta\mu}}\right)^{N-k}A_N$ where $\eta = \underset{i=0,\hdots,N-1}{\min}\,\eta_i\geq\eta_{\mathrm{min}}$.
\end{proof}
When $\mu = 0$, the proof is still valid, and $\tfrac{1}{A_N}=O(N^{-2})$. In particular, we recover the same rates as those of~\cite[Theorem 4.4]{villa2013accelerated} (who used the particular choice $v_{k+1}=\tfrac{y_k-\lambda_k f'(y_k)-x_{k+1}}{\lambda_k}$).
\begin{coro}\label{cor:speed_cvx}
Let $f\in\mathcal{F}_{0,L}(\Rd)$, $g \in \Fmu(\Rd)$ and $F\equiv f+g$. Let $x_0\in\Rd$, $\lambda_0$ be an initial positive step size and sequences (relative error parameters) $\sigma_k = \sigma$, $\zeta_k =\zeta$ with $\sigma,\zeta\in[0,1)$. Let $x_N\in\Rd$ denote the output after $N\in\N^*$ iterations of \Cref{eq:algo2} on $F$ initiated at $x_0\in\Rd$.
\begin{itemize}
    \item We further let $\alpha\in(0,1)$ and $\beta = 1$ be the backtracking parameters, and a sequence (absolute error parameters) $\xi_k=C (k+1)^{-q}$ with $C,q\geq 0$. It holds that 
    \[F(x_N) - F(x_\star)\leq \tfrac{2}{\eta_{\mathrm{min}} N^2}\|x_0-x_\star\|^2+ \left\{\begin{array}{ll}
     2C \tfrac{\eta_{\mathrm{max}}}{\eta_{\mathrm{min}}}\tfrac{\left(\sum_{k=0}^\infty (k+1)^{2-q}\right)}{N^2} & \textrm{if } q > 3,\\
     2C\tfrac{\eta_{\mathrm{max}}}{\eta_{\mathrm{min}}}\tfrac{(1+\ln(N))}{N^2} & \textrm{if } q = 3, \\
    2C\tfrac{\eta_{\mathrm{max}}}{\eta_{\mathrm{min}}}\left(\tfrac{1}{N^2} + \tfrac{1}{(3-q)N^{q-1}} \right) & \textrm{if } 1 < q < 3,\end{array} \right. \]
    with $\eta_{\mathrm{min}} = (1-\zeta^2)\min(\lambda_0,\tfrac{\alpha(1-\sigma^2)}{L})$, $\eta_{\mathrm{max}} = (1-\zeta^2)\max(\lambda_0,\tfrac{(1-\sigma^2)}{L})$ and $x_\star\in\argmin_x F(x)$.
    \item We further let $\alpha\in(0,1)$ and $\beta \geq 1$ be the backtracking parameters, and a sequence $\xi_k=0$ (no absolute error). It holds that
    \[F(x_N) - F(x_\star)\leq \tfrac{2}{\eta_{\mathrm{min}} N^2}\|x_0-x_\star\|^2,\]
    where $\eta_{\mathrm{min}} = (1-\zeta^2)\min(\lambda_0,\tfrac{\alpha(1-\sigma^2)}{L})$ and $x_\star\in\argmin_x F(x)$.
\end{itemize}
\end{coro}
\begin{proof}
Starting from the conclusion of \cref{cor:FB}, we obtain the desired result in the case $\beta =1$ using comparisons of sums with integrals along with the bounds $\tfrac{\eta_{\mathrm{min}}}{4} k^2 \leq A_k \leq \eta_{\mathrm{max}} k^2 $. In the second case, where $\beta\geq 1$ and $\xi_k = 0$, the target result follows from $\tfrac{\eta_{\mathrm{min}}}{4} k^2 \leq A_k$.
\end{proof}

\subsection{ Proof of \cref{thm:deacres-potential}}\label{subsec:proof-FB}

The following proof is presented in a purely algebraic form consisting in a weighted sum of inequalities satisfied by the functions $f$ and $g$ as well as inexactness requirements. Indeed, it has been obtained from a dual certificate of a performance estimation problem (see \cite[Section 3]{barre2020principled_new} for more details on performance estimation in the context of inexact proximal operations). As mentioned in \Cref{sec:intro}, the algebraic equivalences stated below can be verified either by hand or with help of Mathematica notebooks (see \Cref{sec:intro}, \S ``Codes'').
\begin{proof}
    Let $w_{k+1} \in \Rd$ such that $v_{k+1}-\mu x_{k+1} + \mu w_{k+1} \in \partial g(w_{k+1})$. Using \eqref{eq:pdgap-2}, this leads to
    \begin{align*}
        \PDg_{\tfrac{\lambda_{k}}{1+\mu\lambda_{k}} \left(g(\cdot)-\tfrac{\mu}{2}\|\cdot\|^2 \right)}(x_{k+1},v_{k+1}-\mu x_{k+1};\tfrac{y_k-\lambda_kf'(y_k)}{1+\mu\lambda_{k}}) =& \;\tfrac{1}{2(1+\lambda_k\mu)^2}\normsq{x_{k+1} - y_k +\lambda_k(v_{k+1} +f'(y_k))}\\
    &+ \tfrac{\lambda_k}{1+\lambda_k\mu}\Big(g(x_{k+1})-g(w_{k+1}) \\
    &+\tfrac{\mu}{2}\normsq{x_{k+1}-w_{k+1}} - \langle x_{k+1}-w_{k+1},v_{k+1} \rangle \Big).
    \end{align*} 
    The proof consists in performing a weighted sum of the following inequalities:
     \begin{itemize}
        \item strong convexity of $g$ between $w_{k+1}$ and $x_\star$ with weight $\nu_1 = A_{k+1}-A_k$
        \[g(x_\star)\geq g(w_{k+1}) + \langle v_{k+1}-\mu x_{k+1}+\mu w_{k+1}, x_\star-w_{k+1}\rangle + \tfrac{\mu}{2}\normsq{w_{k+1}-x_\star},\]
        \item strong convexity of $g$ between $w_{k+1}$ and $x_{k}$ with weight $\nu_2 = A_k$
        \[g(x_k)\geq g(w_{k+1}) + \langle v_{k+1}-\mu x_{k+1}+\mu w_{k+1}, x_k-w_{k+1}\rangle + \tfrac{\mu}{2}\normsq{w_{k+1}-x_{k}},\]
        \item strong convexity of $g$ between $w_{k+1}$ and $x_{k+1}$ with weight $\nu_3 = A_{k+1}\lambda_k\mu$
        \[g(x_{k+1})\geq g(w_{k+1}) + \langle v_{k+1}-\mu x_{k+1}+\mu w_{k+1}, x_{k+1}-w_{k+1}\rangle + \tfrac{\mu}{2}\normsq{w_{k+1}-x_{k+1}},\]
        \item convexity of $f$ between $y_k$ and $x_\star$ with weight $\nu_4 = A_{k+1}-A_k$
        \[f(x_\star)\geq f(y_{k}) + \langle f'(y_k), x_\star-y_{k}\rangle ,\]
        \item convexity of $f$ between $y_k$ and $x_k$ with weight $\nu_5 = A_k$
        \[f(x_k)\geq f(y_{k}) + \langle f'(y_k), x_k-y_{k}\rangle ,\]
        \item convexity and $\tfrac{1-\sigma_k^2}{\lambda_k}$-smoothness of $f$ between $x_{k+1}$ and $y_k$ required by \eqref{eq:smoothness} with weight $\nu_6=A_{k+1}$
        \[f(y_k)\geq f(x_{k+1}) + \langle f'(x_{k+1}), y_k-x_{k+1}\rangle  + \tfrac{\lambda_k}{2(1-\sigma_k^2)}\normsq{f'(y_k) - f'(x_{k+1})},\]
        \item approximation requirement on $x_{k+1}$ with weight $\nu_7 = \tfrac{A_{k+1}}{\lambda_k}$
        \[\begin{aligned}
        \tfrac{\sigma_k^2}{2}\normsq{x_{k+1}-y_k} +\tfrac{\zeta_k^2\lambda_k^2}{2}\|v_{k+1}
		 +f'(y_k)\|^2 +\tfrac{\lambda_k}{2}\xi_k\geq&\,\lambda_k(1+\lambda_k\mu)\Big(g(x_{k+1})-g(w_{k+1})\\
		 &+\tfrac{\mu}{2}\normsq{x_{k+1}-w_{k+1}}
         - \langle x_{k+1}-w_{k+1},v_{k+1} \rangle \Big)\\
         &+\tfrac{1}{2}\normsq{x_{k+1} - y_k +\lambda_k(v_{k+1}+f'(y_k))}.
        \end{aligned}\]
    \end{itemize}
    The weighted sum can be written as
    \begin{equation}\label{eq:wsum}
    \hspace{-0.1cm}\begin{aligned}
        0 \geq&\, \nu_1\left[  g(w_{k+1})-g(x_\star) + \langle v_{k+1}-\mu x_{k+1}+\mu w_{k+1}, x_\star-w_{k+1}\rangle + \tfrac{\mu}{2}\normsq{w_{k+1}-x_\star}\right]\\
        &+ \nu_2\left[g(w_{k+1}) -g(x_k)+ \langle v_{k+1}-\mu x_{k+1}+\mu w_{k+1}, x_k-w_{k+1}\rangle + \tfrac{\mu}{2}\normsq{w_{k+1}-x_{k}} \right]\\
        &+\nu_3\left[ g(w_{k+1})-g(x_{k+1}) + \langle v_{k+1}-\mu x_{k+1}+\mu w_{k+1}, x_{k+1}-w_{k+1}\rangle + \tfrac{\mu}{2}\normsq{w_{k+1}-x_{k+1}}\right]\\
        &+\nu_4 \left[ f(y_{k})-f(x_\star) + \langle f'(y_k), x_\star-y_{k}\rangle\right]+\nu_5\left[  f(y_{k})-f(x_k) + \langle f'(y_k), x_k-y_{k}\rangle \right]\\
        &+\nu_6\left[ f(x_{k+1}) -f(y_k)+ \langle f'(x_{k+1}), y_k-x_{k+1}\rangle  + \tfrac{\lambda_k}{2(1-\sigma_k^2)}\normsq{f'(y_k) - f'(x_{k+1})}\right]\\
        &+\nu_7\left[ \lambda_k(1+\lambda_k\mu)\Big(g(x_{k+1})-g(w_{k+1})
		 +\tfrac{\mu}{2}\normsq{x_{k+1}-w_{k+1}}
         - \langle x_{k+1}-w_{k+1},v_{k+1} \rangle \Big)\right.\\
         &\left.\quad\quad+\tfrac{1}{2}\normsq{x_{k+1} - y_k +\lambda_k(v_{k+1}+f'(y_k))} -  \tfrac{\sigma_k^2}{2}\normsq{x_{k+1}-y_k} -\tfrac{\zeta_k^2\lambda_k^2}{2}\|v_{k+1}
		 +f'(y_k)\|^2 -\tfrac{\lambda_k}{2}\xi_k\right].
    \end{aligned}
    \end{equation}
    Substituting $y_k$ and $z_{k+1}$ in the weighted sum, that is
    \begin{equation*}
        \begin{array}{ccc}
             y_k &=& x_k+\tfrac{(A_{k+1}-A_k)(1+\mu A_k)}{A_{k+1}+\mu A_k(2 A_{k+1}-A_k)}(z_k-x_{k})  \\
             z_{k+1} &=& z_{k}+\tfrac{A_{k+1}-A_k}{1+\mu A_{k+1}}\left(\mu(x_{k+1}-z_k) -(v_{k+1}+f'(y_k))\right),
        \end{array}
    \end{equation*}
    \eqref{eq:wsum} is equivalently reformulated as
    \begin{align*}
        &A_{k+1}(F(x_{k+1}) - F(x_\star)) + \tfrac{1+\mu A_{k+1}}{2}\normsq{z_{k+1}-x_\star} \\
        &\leq A_{k}(F(x_{k}) - F(x_\star)) + \tfrac{1+\mu A_{k}}{2}\normsq{z_{k}-x_\star} +\tfrac{A_{k+1}}{2}\xi_k\\
        &\,\,\,\,\,- \tfrac{A_k(A_{k+1}-A_k)\mu (1 + A_k \mu)}{2(A_{k+1} + A_k(2A_{k+1}-A_k)\mu)}\normsq{x_k-z_k}\\
        &\,\,\,\,\,- \tfrac{A_{k+1}\lambda_k}{2(1-\sigma_k^2)}\normsq{f'(x_{k+1})-f'(y_k)+(1-\sigma_k^2)\tfrac{y_k-x_{k+1}}{\lambda_k}}\\
        &\,\,\,\,\,- \tfrac{\mu\left(A_{k+1}+A_k(2A_{k+1}-A_k)\mu \right)}{2(1+A_{k+1}\mu)}\normsq{x_{k+1}-y_k + \tfrac{(A_{k+1}-A_k)^2}{A_{k+1}+A_k(2A_{k+1}-A_k)\mu}(v_{k+1}+f'(y_k))}\\
        &\,\,\,\,\,+ A_{k+1}\tfrac{A_{k+1}(A_{k+1}-\eta_k)-A_k(1+\eta_k\mu)(2A_{k+1}-A_k)}{2(A_{k+1}+A_k(2A_{k+1}-A_k)\mu)}\normsq{v_{k+1}+f'(y_k)}\\
        &\leq A_{k}(F(x_{k}) - F(x_\star)) + \tfrac{1+\mu A_{k}}{2}\normsq{z_{k}-x_\star}+\tfrac{A_{k+1}}{2}\xi_k\\
        &\,\,\,\,\,+ A_{k+1}\tfrac{A_{k+1}(A_{k+1}-\eta_k)-A_k(1+\eta_k\mu)(2A_{k+1}-A_k)}{2(A_{k+1}+A_k(2A_{k+1}-A_k)\mu)}\normsq{v_{k+1}+f'(y_k)}\\
        &=A_{k}(F(x_{k}) - F(x_\star)) + \tfrac{1+\mu A_{k}}{2}\normsq{z_{k}-x_\star}+\tfrac{A_{k+1}}{2}\xi_k,
    \end{align*}
    where the inequality in the second to last line comes from the fact that factors in front of three squared Euclidean norms are nonpositive. In addition, the last equality follows from the particular choice of $A_{k+1}$ satisfies
    \[A_{k+1}(A_{k+1}-\eta_k) - A_k(1+\eta_k\mu)(2A_{k+1}-A_k) = 0,\]
    which implies that the factors in front of the last squared Euclidean norm vanishes.
\end{proof}

\section{Numerical examples}\label{s:numerics}

In this section, we present a few numerical experiments illustrating the behavior of the accelerated inexact forward backward method (\Cref{eq:algo2}) on two convex problems. More precisely, we applied the method to a factorization problem and to a total variation problem.%on the ``a1a'' dataset from the LIBSVM library \citep{CC01a}, and to a total variation problem.

In both cases, we use a Tikhonov regularization, improving the conditioning and rendering the problems strongly convex, and illustrate the numerical performances of the algorithm with different tunings, including in the purely relative ($\xi_k=0$) and absolute accuracy ($\sigma_k=\zeta_k=0$) setups, as well as the influence of the knowledge of strong convexity parameter.

\subsection{ Factorization problem} Our first numerical experiment is a CUR-like factorization problem, introduced in~\citet{mairal2011convex}. It consists, given a matrix $W \in \mathrm{R}^{m\times p}$, in solving the minimization problem
\[ \min_X\,F(X) \equiv \underbrace{\tfrac{1}{2}\|W-WXW\|^2_{F}}_{f(X)} + \underbrace{\lambda_{\mathrm{row}}\sum_{i=1}^{n_r}\|X^i\|_2 + \lambda_{\mathrm{col}}\sum_{j=1}^{n_c}\|X_j\|_2 + \tfrac{\mu_\mathrm{reg}}{2}\|X\|^2_F}_{g(X)},  \]
where $\|\cdot\|_{F}$ is the Frobenius norm, and where $X^i$ and $X_j$ respectively denote the $i$th row and the $j$th column of the matrix $X$. This problem has already been used in~\citet{schmidt2011convergence} for illustrating convergence guarantees of an inexact accelerated proximal gradient method with absolute errors. As in~\citet{schmidt2011convergence}, we use an inexact version of the proximal operator of the regularization part, which we solve via a dual block coordinate ascent method~\citep{jenatton2010proximal} (i.e., we solve the dual of the proximal problem). Our implementation (see link in \S Codes from Section~\ref{sec:intro}) is based on that of~\citet{schmidt2011convergence}, and our experiments are done on the ``a1a'' dataset from the LIBSVM library \citep{CC01a}. The corresponding matrix $W$ is normalized for having zero mean and unit norm. We also impose $\lambda_{\text{col}} = \sqrt{\tfrac{p}{m}}\lambda_{\text{row}}$ for having a similar scaling for the row and column regularization parameters. The choice of the error criteria and regularization parameters is detailed in \cref{fig:CUR} where we plot gaps between the objective function values at the iterates of \Cref{eq:algo2} and the optimal objective value versus the number of iteration of \Cref{eq:algo2} (left) and versus the total number of dual block coordinate ascent iterations (right).

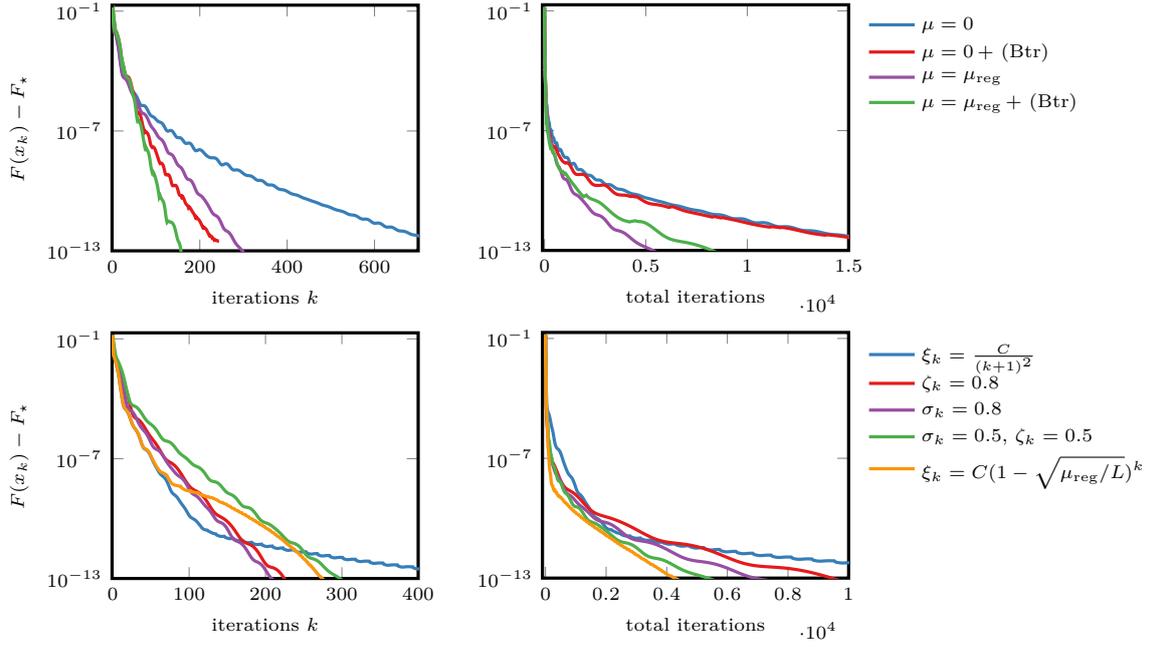
\begin{figure}[!ht]
\centering
\begin{tabular}{ll}
\begin{tikzpicture}
			\begin{semilogyaxis}[legend pos=outer north east, legend style={draw=none},legend cell align={left},plotOptions, ymin=1e-13, ymax=2e-1,xmin=-1,xmax=700,width=.33\linewidth, ylabel={$F(x_{k})-F_\star$}, xlabel={iterations $k$}]
			%\addplot [color=colorP1, dashed, domain=1:50, samples=20] {2/x};
			%\addlegendentry{$O(\tfrac{1}{k})$}
			\addplot [colorP1] table [x=N,y=f,col sep=space] {data/CUR/CUR_a1a_L_1.331492e-01_mu_2.662985e-04_sigma_8.000000e-01_zeta_0.000000e+00_xi_zero_maxiter_20000_inneriter_20000_no_mu.txt};
			\addplot [colorP2] table [x=N,y=f,col sep=space] {data/CUR/CUR_a1a_L_1.331492e-01_mu_2.662985e-04_sigma_8.000000e-01_zeta_0.000000e+00_xi_zero_maxiter_20000_inneriter_20000_no_mu_backtrack.txt};
			\addplot [colorP3] table [x=N,y=f,col sep=space] {data/CUR/CUR_a1a_L_1.331492e-01_mu_2.662985e-04_sigma_8.000000e-01_zeta_0.000000e+00_xi_zero_maxiter_20000_inneriter_20000.txt};
			\addplot [colorP4] table [x=N,y=f,col sep=space] {data/CUR/CUR_a1a_L_1.331492e-01_mu_2.662985e-04_sigma_8.000000e-01_zeta_0.000000e+00_xi_zero_maxiter_20000_inneriter_20000_backtrack.txt};
			\end{semilogyaxis}
\end{tikzpicture}
    & \begin{tikzpicture}
			\begin{semilogyaxis}[legend pos=outer north east, legend style={draw=none},legend cell align={left},plotOptions, ymin=1e-13, ymax=2e-1,xmin=-100,xmax=15000,width=.33\linewidth, xlabel={total iterations}]
			%\addplot [color=colorP1, dashed, domain=1:50, samples=20] {2/x};
			%\addlegendentry{$O(\tfrac{1}{k})$}
			
			\addplot [colorP1] table [x=cs,y=f,col sep=space] {data/CUR/CUR_a1a_L_1.331492e-01_mu_2.662985e-04_sigma_8.000000e-01_zeta_0.000000e+00_xi_zero_maxiter_20000_inneriter_20000_no_mu.txt};
			\addplot [colorP2] table [x=cs,y=f,col sep=space] {data/CUR/CUR_a1a_L_1.331492e-01_mu_2.662985e-04_sigma_8.000000e-01_zeta_0.000000e+00_xi_zero_maxiter_20000_inneriter_20000_no_mu_backtrack.txt};
			\addplot [colorP3] table [x=cs,y=f,col sep=space] {data/CUR/CUR_a1a_L_1.331492e-01_mu_2.662985e-04_sigma_8.000000e-01_zeta_0.000000e+00_xi_zero_maxiter_20000_inneriter_20000.txt};
			\addplot [colorP4] table [x=cs,y=f,col sep=space] {data/CUR/CUR_a1a_L_1.331492e-01_mu_2.662985e-04_sigma_8.000000e-01_zeta_0.000000e+00_xi_zero_maxiter_20000_inneriter_20000_backtrack.txt};
			\addlegendentry{$\mu = 0$}
			\addlegendentry{$\mu = 0 \,+ $ (Btr)}
			\addlegendentry{$\mu = \mu_{\mathrm{reg}} $}
			\addlegendentry{$\mu = \mu_{\mathrm{reg}} \,+ $ (Btr)}
			\end{semilogyaxis}
\end{tikzpicture}
\\
\begin{tikzpicture}
			\begin{semilogyaxis}[legend pos=outer north east, legend style={draw=none},legend cell align={left},plotOptions, ymin=1e-13,ymax=2e-1,xmin=-1,xmax=400,width=.33\linewidth, ylabel={$F(x_{k})-F_\star$}, xlabel={iterations $k$}]
			%\addplot [color=colorP1, dashed, domain=1:50, samples=20] {2/x};
			%\addlegendentry{$O(\tfrac{1}{k})$}
			
			\addplot [colorP1] table [x=N,y=f,col sep=space] {data/CUR/CUR_a1a_L_1.331492e-01_mu_2.662985e-04_sigma_0.000000e+00_zeta_0.000000e+00_xi_poly2_maxiter_20000_inneriter_20000.txt};
			\addplot [colorP2] table [x=N,y=f,col sep=space] {data/CUR/CUR_a1a_L_1.331492e-01_mu_2.662985e-04_sigma_0.000000e+00_zeta_8.000000e-01_xi_zero_maxiter_20000_inneriter_20000.txt};
			\addplot [colorP3] table [x=N,y=f,col sep=space] {data/CUR/CUR_a1a_L_1.331492e-01_mu_2.662985e-04_sigma_5.000000e-01_zeta_5.000000e-01_xi_zero_maxiter_20000_inneriter_20000.txt};
			\addplot [colorP4] table [x=N,y=f,col sep=space] {data/CUR/CUR_a1a_L_1.331492e-01_mu_2.662985e-04_sigma_8.000000e-01_zeta_0.000000e+00_xi_zero_maxiter_20000_inneriter_20000.txt};
			\addplot [colorP5] table [x=N,y=f,col sep=space] {data/CUR/CUR_a1a_L_1.331492e-01_mu_2.662985e-04_sigma_0.000000e+00_zeta_0.000000e+00_xi_lin_maxiter_20000_inneriter_20000.txt};
			\end{semilogyaxis}
\end{tikzpicture}
    & \begin{tikzpicture}
			\begin{semilogyaxis}[legend pos=outer north east, legend style={draw=none},legend cell align={left},plotOptions, ymin=1e-13, ymax=2e-1,xmin=-100,xmax=10000,width=.33\linewidth, xlabel={total iterations}]
			%\addplot [color=colorP1, dashed, domain=1:50, samples=20] {2/x};
			%\addlegendentry{$O(\tfrac{1}{k})$}
			
			\addplot[color=colorP1] table [x=cs,y=f,col sep=space] {data/CUR/CUR_a1a_L_1.331492e-01_mu_2.662985e-04_sigma_0.000000e+00_zeta_0.000000e+00_xi_poly2_maxiter_20000_inneriter_20000.txt};
			\addplot [colorP2] table [x=cs,y=f,col sep=space] {data/CUR/CUR_a1a_L_1.331492e-01_mu_2.662985e-04_sigma_0.000000e+00_zeta_8.000000e-01_xi_zero_maxiter_20000_inneriter_20000.txt};
			\addplot [colorP3] table [x=cs,y=f,col sep=space] {data/CUR/CUR_a1a_L_1.331492e-01_mu_2.662985e-04_sigma_5.000000e-01_zeta_5.000000e-01_xi_zero_maxiter_20000_inneriter_20000.txt};
			\addplot [colorP4] table [x=cs,y=f,col sep=space] {data/CUR/CUR_a1a_L_1.331492e-01_mu_2.662985e-04_sigma_8.000000e-01_zeta_0.000000e+00_xi_zero_maxiter_20000_inneriter_20000.txt};
			\addplot [colorP5] table [x=cs,y=f,col sep=space] {data/CUR/CUR_a1a_L_1.331492e-01_mu_2.662985e-04_sigma_0.000000e+00_zeta_0.000000e+00_xi_lin_maxiter_20000_inneriter_20000.txt};
			\addlegendentry{$\xi_k = \tfrac{C}{(k+1)^2}$}
			\addlegendentry{$\zeta_k = 0.8$}
			\addlegendentry{$\sigma_k = 0.8$}
			\addlegendentry{$\sigma_k = 0.5$, $\zeta_k=0.5$}
			\addlegendentry{$\xi_k = C(1-\sqrt{\mu_{\mathrm{reg}}/L})^{k}$}
			\end{semilogyaxis}
\end{tikzpicture}

\end{tabular}
		\caption{\Cref{eq:algo2} on CUR factorization. The initial step size is set to $\lambda_0 = \tfrac{1-\sigma_0^2}{L}$, initial $L$ to $\norm{W}^4$, $\lambda_\mathrm{row} = \lambda_\mathrm{col}\sqrt{m/p} = 2.10^{-3}$ ($\sim 30\%$ nonzero coefficients in the solution) and $\mu_{\mathrm{reg}} = 2.10^{-3}L$. Top: $\sigma_k = 0.8$ , $\zeta_k = 0$ and $\xi_k = 0$. Bottom: accelerated inexact forward-backward with $\mu=\mu_{\mathrm{reg}}$ and no backtracking.  When backtracking is used, $\alpha$ is set to $\tfrac12$ and $\beta$ to $1.1$. ``Total iterations'' refers to total the number of block coordinate ascent iterations used in the subroutine that computes the proximal steps approximately. $F_\star$ is approximated by the smallest objective values encountered in $2.10^4$ total iterations of block coordinate ascent.} \label{fig:CUR}
		\vspace{-0.0cm}
\end{figure}

\subsection{ Total variation regularization} In this section, we compare the behaviors of the accelerated inexact forward backward method (\Cref{eq:algo2}) with different tunings, on the classical problem of deblurring through total variation regularization~\citep{rudin1992nonlinear,rudin1994total,wang2008new}. Given a blurred image $Y \in \mathrm{R}^{n\times n}$ and a blurring operator $A$, the problem consists in solving 
\[\min_X \,F(X) \equiv \underbrace{\tfrac{1}{2}\|AX-Y\|_F^2}_{f(X)} + \underbrace{\lambda_{\text{reg}}\sum_{i,j = 0}^{n}\|(\nabla X)_{i,j}\|_2  + \tfrac{\mu_{\mathrm{reg}}}{2}\|X\|^2_F}_{g(X)},\]
where $\nabla$ is the discrete gradient of an image, see e.g.,~\cite[Equation (2.4)]{chambolle2016introduction}. One way of dealing with this problem is to approximate the proximal operator of the discrete total variation plus the Tikhonov regularization. As in~\citet{villa2013accelerated,millan2019inexact}, we apply FISTA~\citep{beck2009fast} on the dual of the proximal subproblem (which is provided e.g., in~\cite[Example 3.1]{chambolle2016introduction}), which we use in the accelerated inexact forward-backward method. 

In the experiments $Y$ is the popular $256\times256$ greyscale boat image (see e.g.,~\url{http://sipi.usc.edu/database/}). We blur $Y$ via a $5 \times 5$ box blur kernel $A$, and add a Gaussian noise of standard deviation $0.01$ times the mean of the blurred image and zero mean to the picture. Some results are detailed in \cref{fig:TV} where we plot gaps between the objective function values at the iterates of \Cref{eq:algo2} and the optimal objective value versus the number of iterations of \Cref{eq:algo2} (left) and versus the total number iterations of FISTA on the dual subproblem (right).

\begin{figure}[!ht]
\centering
\begin{tabular}{ll}
\begin{tikzpicture}
			\begin{semilogyaxis}[legend pos=outer north east, legend style={draw=none},legend cell align={left},plotOptions, ymin=0.0001, ymax=1e7,xmin=-1,xmax=300,width=.33\linewidth, ylabel={$F(x_{k})-F_\star$}, xlabel={iterations $k$}]
			%\addplot [color=colorP1, dashed, domain=1:50, samples=20] {2/x};
			%\addlegendentry{$O(\tfrac{1}{k})$}
			
			\addplot [colorP1] table [x=N,y=f,col sep=space] {data/TV_boat_L_1.000000e+00_mu_1.000000e-02_sigma_8.000000e-01_zeta_0.000000e+00_xi_zero_maxiter_70000_inneriter_10000_no_mu.txt};
			\addplot [colorP2] table [x=N,y=f,col sep=space] {data/TV_boat_L_1.000000e+00_mu_1.000000e-02_sigma_8.000000e-01_zeta_0.000000e+00_xi_zero_maxiter_70000_inneriter_10000_no_mu_backtrack.txt};
			\addplot [colorP3] table [x=N,y=f,col sep=space] {data/TV_boat_L_1.000000e+00_mu_1.000000e-02_sigma_8.000000e-01_zeta_0.000000e+00_xi_zero_maxiter_70000_inneriter_10000.txt};
			\addplot [colorP4] table [x=N,y=f,col sep=space] {data/TV_boat_L_1.000000e+00_mu_1.000000e-02_sigma_8.000000e-01_zeta_0.000000e+00_xi_zero_maxiter_70000_inneriter_10000_backtrack.txt};
			\end{semilogyaxis}
\end{tikzpicture}
    & \begin{tikzpicture}
			\begin{semilogyaxis}[legend pos=outer north east, legend style={draw=none},legend cell align={left},plotOptions, ymin=0.0001, ymax=1e7,xmin=-500,xmax=70000,width=.33\linewidth, xlabel={total iterations} ]
			%\addplot [color=colorP1, dashed, domain=1:50, samples=20] {2/x};
			%\addlegendentry{$O(\tfrac{1}{k})$}
			
			\addplot [colorP1] table [x=cs,y=f,col sep=space] {data/TV_boat_L_1.000000e+00_mu_1.000000e-02_sigma_8.000000e-01_zeta_0.000000e+00_xi_zero_maxiter_70000_inneriter_10000_no_mu.txt};
			\addplot [colorP2] table [x=cs,y=f,col sep=space] {data/TV_boat_L_1.000000e+00_mu_1.000000e-02_sigma_8.000000e-01_zeta_0.000000e+00_xi_zero_maxiter_70000_inneriter_10000_no_mu_backtrack.txt};
			\addplot [colorP3] table [x=cs,y=f,col sep=space] {data/TV_boat_L_1.000000e+00_mu_1.000000e-02_sigma_8.000000e-01_zeta_0.000000e+00_xi_zero_maxiter_70000_inneriter_10000.txt};
			\addplot [colorP4] table [x=cs,y=f,col sep=space] {data/TV_boat_L_1.000000e+00_mu_1.000000e-02_sigma_8.000000e-01_zeta_0.000000e+00_xi_zero_maxiter_70000_inneriter_10000_backtrack.txt};
			
			\addlegendentry{$\mu = 0$}
			\addlegendentry{$\mu = 0 \,+ $ (Btr)}
			\addlegendentry{$\mu = \mu_{\mathrm{reg}} $}
			\addlegendentry{$\mu = \mu_{\mathrm{reg}} \,+ $ (Btr)}
			\end{semilogyaxis}
\end{tikzpicture}
\\
\begin{tikzpicture}
			\begin{semilogyaxis}[legend pos=outer north east, legend style={draw=none},legend cell align={left},plotOptions, ymin=0.0001, ymax=1e7,xmin=-1,xmax=300,width=.33\linewidth, ylabel={$F(x_{k})-F_\star$}, xlabel={iterations $k$}]
			%\addplot [color=colorP1, dashed, domain=1:50, samples=20] {2/x};
			%\addlegendentry{$O(\tfrac{1}{k})$}
			
			\addplot [colorP1] table [x=N,y=f,col sep=space] {data/TV_boat_L_1.000000e+00_mu_1.000000e-02_sigma_0.000000e+00_zeta_0.000000e+00_xi_poly2_maxiter_70000_inneriter_10000_backtrack.txt};
			\addplot [colorP2] table [x=N,y=f,col sep=space] {data/TV_boat_L_1.000000e+00_mu_1.000000e-02_sigma_0.000000e+00_zeta_8.000000e-01_xi_zero_maxiter_70000_inneriter_10000_backtrack.txt};
			\addplot [colorP3] table [x=N,y=f,col sep=space] {data/TV_boat_L_1.000000e+00_mu_1.000000e-02_sigma_8.000000e-01_zeta_0.000000e+00_xi_zero_maxiter_70000_inneriter_10000_backtrack.txt};
			\addplot [colorP4] table [x=N,y=f,col sep=space] {data/TV_boat_L_1.000000e+00_mu_1.000000e-02_sigma_5.000000e-01_zeta_5.000000e-01_xi_zero_maxiter_70000_inneriter_10000_backtrack.txt};

			\addplot [colorP5] table [x=N,y=f,col sep=space] {data/TV_boat_L_1.000000e+00_mu_1.000000e-02_sigma_0.000000e+00_zeta_0.000000e+00_xi_lin_maxiter_70000_inneriter_10000_backtrack.txt};
			\end{semilogyaxis}
\end{tikzpicture}
    & \begin{tikzpicture}
			\begin{semilogyaxis}[legend pos=outer north east, legend style={draw=none},legend cell align={left},plotOptions, ymin=0.0001, ymax=1e7,xmin=-500,xmax=70000,width=.33\linewidth, xlabel={total iterations}]
			%\addplot [color=colorP1, dashed, domain=1:50, samples=20] {2/x};
			%\addlegendentry{$O(\tfrac{1}{k})$}
			\addplot [colorP1] table [x=cs,y=f,col sep=space] {data/TV_boat_L_1.000000e+00_mu_1.000000e-02_sigma_0.000000e+00_zeta_0.000000e+00_xi_poly2_maxiter_70000_inneriter_10000_backtrack.txt};
			
			\addplot [colorP2] table [x=cs,y=f,col sep=space] {data/TV_boat_L_1.000000e+00_mu_1.000000e-02_sigma_0.000000e+00_zeta_8.000000e-01_xi_zero_maxiter_70000_inneriter_10000_backtrack.txt};
			\addplot [colorP3] table [x=cs,y=f,col sep=space] {data/TV_boat_L_1.000000e+00_mu_1.000000e-02_sigma_8.000000e-01_zeta_0.000000e+00_xi_zero_maxiter_70000_inneriter_10000_backtrack.txt};
			\addplot [colorP4] table [x=cs,y=f,col sep=space] {data/TV_boat_L_1.000000e+00_mu_1.000000e-02_sigma_5.000000e-01_zeta_5.000000e-01_xi_zero_maxiter_70000_inneriter_10000_backtrack.txt};
			
			\addplot [colorP5] table [x=cs,y=f,col sep=space] {data/TV_boat_L_1.000000e+00_mu_1.000000e-02_sigma_0.000000e+00_zeta_0.000000e+00_xi_lin_maxiter_70000_inneriter_10000_backtrack.txt};
			\addlegendentry{$\xi_k = \tfrac{C}{(k+1)^2}$}
			\addlegendentry{$\zeta_k = 0.8$}
			\addlegendentry{$\sigma_k = 0.8$}
			\addlegendentry{$\sigma_k = 0.5$, $\zeta_k=0.5$}
			\addlegendentry{$\xi_k = C(1-\sqrt{\mu_{\mathrm{reg}}/L})^{k}$}
			\end{semilogyaxis}
\end{tikzpicture}

\end{tabular}
		\caption{\Cref{eq:algo2} on TV regularization. The initial step size is set to $\lambda_0 = \tfrac{1-\sigma_0^2}{L}$, the initial $L$ and $\lambda_{\mathrm{reg}}$ to $1$ and $\mu_{reg}$ to $10^{-2} L$. Top: $\sigma_k = 0.8$ , $\zeta_k = 0$ and $\xi_k = 0$. Bottom: accelerated inexact forward-backward with $\mu=\mu_{\mathrm{reg}}$ and backtracking. When backtracking is used, $\alpha$ is set to $\tfrac12$ and $\beta$ to $1.1$. ``Total iterations'' refers to total the number of FISTA iterations used in the subroutine that computes the proximal steps approximately.  $F_\star$ is approximated by the smallest objective values encountered in $2.10^4$ total FISTA iterations.} \label{fig:TV}
		\vspace{-0.5cm}
\end{figure}
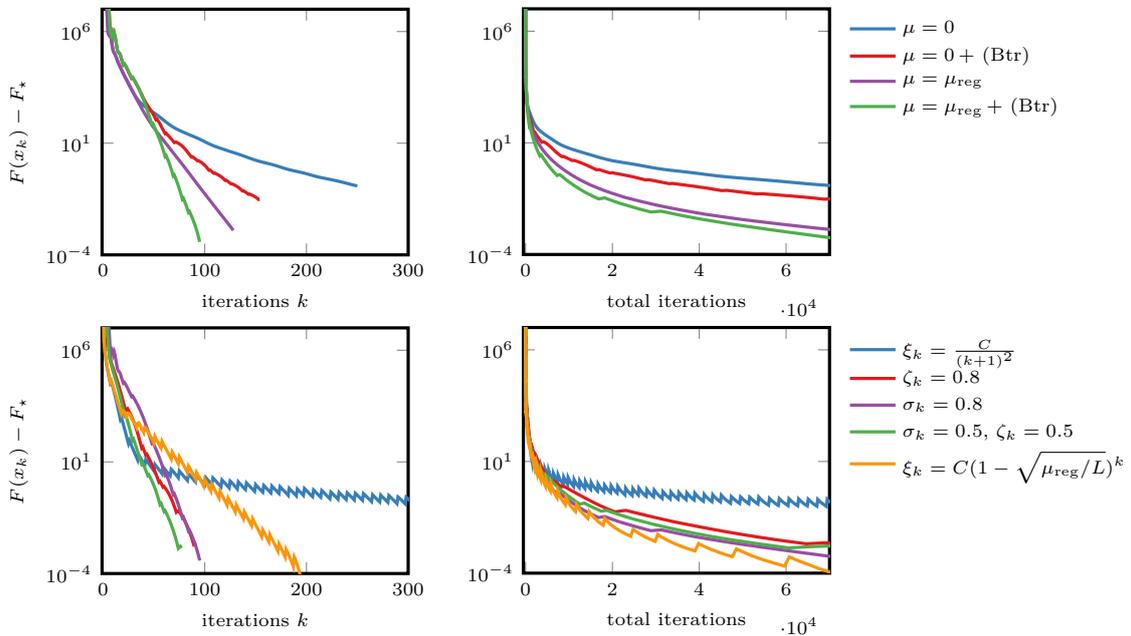

\newpage
\section{An accelerated hybrid proximal extragradient method}\label{s:ahpe}
In this section, we provide an improved analysis for the specific case $f=0$ (no smooth convex term in~\eqref{eq:opt}). This type of methods is often used as a \emph{globalization} strategy for higher-order methods, see~\citet{monteiro2013accelerated}. The version presented in this section allows exploiting the possible strong convexity of the objective, which was not incorporated in previous versions of the method, to the best of our knowledge.

% Using this approximation criterion, we define the following inexact accelerated proximal methods for strongly convex objectives.
% {\color{teal}+say it allows exploiting str cvx objectives}

\subsection{ Algorithm}
\begin{oframed}
	\textbf{An accelerated hybrid proximal extragradient method} (\Cref{eq:ahpe}) \label[algorithm]{eq:ahpe}
	\begin{itemize}
		\item[] {\bf{}Input:} \begin{itemize}
		    \item Objective function: $g\in\Fmu(\Rd)$.
		    \item Initial point: $x_0\in\mathbb{R}^d$.
		    \item Step sizes: $\{\lambda_k\}_k$ with $\lambda_k >0$.
		    \item Tolerance parameters: sequence $\{\sigma_k\}_k$ with  $\sigma_k\in[0,1]$. 
		\end{itemize}
		\item[] {\bf{}Initialization:} $z_0=x_0$, $A_0=0$. 
		\item[] {\bf{}Run:} For $k=0,1,\hdots$:
		\begin{equation}\notag
			\begin{aligned}
			A_{k+1}&=A_k+(2(1-\sigma_k)+\lambda_{k}\mu)\lambda_{k}\tfrac{1 + 2 A_k\mu+ \sqrt{1 +4A_k(1+A_k\mu)\tfrac{(1+\lambda_k\mu)^2 - \sigma_k(\sigma_k+\lambda_k\mu)}{(2(1-\sigma_k)+\lambda_{k}\mu)\lambda_{k}}}}{{2(1-\sigma_k^2 + \lambda_{k}\mu\sigma_k)}}\\
            y_{k}&=x_k+\tfrac{(A_{k+1}-A_k)(1+\mu A_k)}{A_{k+1}+\mu A_k(2 A_{k+1}-A_k)}(z_k-x_{k})\\
            \varepsilon_k &= \tfrac{\sigma_k^2}{2(1+\lambda_k \mu)^2} \|x_{k+1}-y_k\|^2\\
			(x_{k+1},\;v_{k+1})&\approx_{\varepsilon_k,\mu}\left(\prox_{\lambda_{k} g}(y_{k}),\prox_{g^*/\lambda_{k}}(\tfrac{y_{k}}{\lambda_k})\right) \\
			{z}_{k+1}&=z_{k}+\tfrac{A_{k+1}-A_k}{1+\mu A_{k+1}}\left(\mu(x_{k+1}-z_k) -v_{k+1} \right)\\
			\end{aligned}
			\end{equation}
			\item[] {\bf{}Output:} $x_{k+1}$
		\end{itemize}  
	\end{oframed}
When $\mu=0$ and $\sigma_k$ is fixed, this method actually reduces to the optimized relatively inexact proximal point algorithm from \cite[(ORIPPA)]{barre2020principled_new}. In this case, the growth rate of the sequence $\{A_k\}_k$ is essentially
\[ A_{k}\geq \tfrac{1}{4}\left(\sum_{i=0}^{k-1} \sqrt{\tfrac{2\lambda_i}{(1+\sigma_i)}}\right)^2 \geq O(k^2)\quad \text{for } k \geq 1,\]
when the parameters $\{\lambda_k\}_k$, $\{\sigma_k\}_k$ are well chosen (e.g., constant parameters). When $\mu>0$, the sequence $\{A_k\}_k$ grows as
\begin{align*}
A_{k+1} &\geq A_k\left(1 + \tfrac{2(1-\sigma_k)+\lambda_k\mu}{1-\sigma_k^2+\lambda_k\mu\sigma_k}\left(\lambda_k\mu + \sqrt{\lambda_k\mu\tfrac{(1+\lambda_k\mu)^2 - \sigma_k(\sigma_k+\lambda_k\mu)}{(2(1-\sigma_k)+\lambda_{k}\mu)}}\right) \right)\\
&= A_k\left( \tfrac{(1+\lambda_k\mu)^2 - \sigma_k(\sigma_k+\lambda_k\mu)+\sqrt{\lambda_k\mu(2(1-\sigma_k)+\lambda_{k}\mu)((1+\lambda_k\mu)^2 - \sigma_k(\sigma_k+\lambda_k\mu))}}{1-\sigma_k^2+\lambda_k\mu\sigma_k}\right)\\
&= A_k\left(\tfrac{(1+\lambda_k\mu)^2 - \sigma_k(\sigma_k+\lambda_k\mu) - \lambda_k\mu(2(1-\sigma_k)+\lambda_{k}\mu)}{1-\sigma_k^2+\lambda_k\mu\sigma_k} \right)\big/\left(1 - \sqrt{\tfrac{\lambda_k \mu (2(1-\sigma_k) + \lambda_k\mu)}{(1+\lambda_k \mu)^2 - \sigma_k(\sigma_k+\lambda_k \mu)}}\right)\\
    &= A_k\big/\left(1 - \sqrt{\tfrac{\lambda_k \mu (2(1-\sigma_k) + \lambda_k\mu)}{(1+\lambda_k \mu)^2 - \sigma_k(\sigma_k+\lambda_k \mu)}}\right),
\end{align*}
with $A_1 = \lambda_0\tfrac{2(1-\sigma_0)+\lambda_0\mu}{(1-\sigma_0^2+\lambda_0\mu\sigma_0)}\geq \lambda_0$.
In particular we recover the rate of the inexact accelerated forward-backward method when $\sigma_k = 1$. In addition, we notice that 
\[ 1 - \sqrt{\tfrac{\lambda_k \mu (2 + \lambda_k\mu - 2\sigma_k)}{(1+\lambda_k \mu)^2 - \sigma_k(\sigma_k+\lambda_k \mu)}} \sim 1-\sqrt{\tfrac{2}{1+\sigma_k}\lambda_k \mu},\]
when $\lambda_k \mu \ll 0$.

\begin{theo}\label{thm:ahpe}
    Let $g\in \Fmu(\Rd)$, $k\geq 0$, a parameter $\sigma_k\in[0,1]$ and some $\lambda_k > 0$. For any $x_k,z_k\in\Rd$ and $A_k\geq 0$, it holds that
    \[A_{k+1}(g(x_{k+1}) - g(x_\star))+\tfrac{1+\mu A_{k+1}}{2}\|z_{k+1}-x_\star\|^2 \leq  A_{k}(g(x_{k}) - g(x_\star))+\tfrac{1+\mu A_k}{2}\|z_{k}-x_\star\|^2,\]
    with $x_\star\in \argmin_x g(x)$, and where $z_{k+1}$, $x_{k+1}$ are constructed by one iteration of \Cref{eq:ahpe}.
\end{theo}
\begin{proof}
    The proof of this theorem is deferred to \cref{subsec:proof-AHPE}
\end{proof}
Just as for the its forward-backward version, one can obtain a final worst-case guarantee driven by the growth rate of the sequence $\{A_k\}_k$. 
\begin{coro}
Let $g \in \Fmu(\Rd)$, $\{\lambda_k\}_k$ be a sequence of positive parameters, and a sequence (relative error parameters) $\{\sigma_k\}_k$ satisfying $\sigma_k\in[0,1]$. Let $x_N\in\Rd$ be the output after $N\in\N^*$ iterations of \Cref{eq:ahpe} on $g$ initiated at $x_0\in\Rd$, it holds that
\[ g(x_N) - g(x_\star) \leq \tfrac{1}{2A_N}\normsq{x_0-x_\star},\]
where $x_\star \in \argmin_x g(x)$.
\end{coro}
\begin{proof}
    The proof follows from the same lines as that of~\cref{cor:FB}, using \cref{thm:ahpe} instead of \cref{thm:deacres-potential}.
\end{proof}
% \AT{ajouter une remarque pour dire que le taux de convergence garanti est donc l'inverse du growth rate de $\{A_k\}_k$? pas besoin de tres formel, se ref. a la section prec.}
%{\color{teal}add corollary with worst-case bound ++ speed of $A_k$}
\subsection{ Proof of \cref{thm:ahpe}}\label{subsec:proof-AHPE}
The proof follows the same structure as that of in \cref{subsec:proof-FB}, and simply consists in reformulating a weighted sum of inequalities.
\begin{proof}
    First we consider $\sigma_k \in (0,1]$ as the case $\sigma_k=0$ requires a particular treatment.
    Let $w_{k+1} \in \Rd$ such that $v_{k+1}-\mu x_{k+1} + \mu w_{k+1} \in \partial g(w_{k+1})$. Using \eqref{eq:pdgap-2}, this leads to
    \begin{align*}
        \PDg_{\tfrac{\lambda_{k}}{1+\mu\lambda_{k}} \left(g(\cdot)-\tfrac{\mu}{2}\|\cdot\|^2 \right)}(x_{k+1},v_{k+1}-\mu x_{k+1};\tfrac{y_k}{1+\mu\lambda_{k}}) =& \;\tfrac{1}{2(1+\lambda_k\mu)^2}\normsq{x_{k+1} - y_k +\lambda_k v_{k+1}}\\
    &+ \tfrac{\lambda_k}{1+\lambda_k\mu}\Big(g(x_{k+1})-g(w_{k+1})+\tfrac{\mu}{2}\normsq{x_{k+1}-w_{k+1}} \\
    & - \langle x_{k+1}-w_{k+1},v_{k+1} \rangle \Big).
    \end{align*} 
    The proof consists in performing a weighted sum of the following inequalities:
    \begin{itemize}
        \item strong convexity between $w_{k+1}$ and $x_\star$ with weight $\nu_1 = A_{k+1} - A_k$
        \[g(x_\star)\geq g(w_{k+1}) + \langle v_{k+1}-\mu x_{k+1}+\mu w_{k+1}, x_\star-w_{k+1}\rangle + \tfrac{\mu}{2}\normsq{w_{k+1}-x_\star},\]
        \item strong convexity between $w_{k+1}$ and $x_{k}$ with weight $\nu_2 = A_k$
        \[g(x_k)\geq g(w_{k+1}) + \langle v_{k+1}-\mu x_{k+1}+\mu w_{k+1}, x_k-w_{k+1}\rangle + \tfrac{\mu}{2}\normsq{w_{k+1}-x_{k}},\]
        \item strong convexity between $w_{k+1}$ and $x_{k+1}$ with weight $\nu_3 = \tfrac{A_{k+1}(1-\sigma_k+\lambda_k\mu)}{\sigma_k}$
        \[g(x_{k+1})\geq g(w_{k+1}) + \langle v_{k+1}-\mu x_{k+1}+\mu w_{k+1}, x_{k+1}-w_{k+1}\rangle + \tfrac{\mu}{2}\normsq{w_{k+1}-x_{k+1}},\]
        \item approximation requirement on $x_{k+1}$ with weight $\nu_4 = \tfrac{A_{k+1}}{\lambda_k\sigma_k}$
        \[\begin{aligned}
        \tfrac{\sigma_k^2}{2}\normsq{x_{k+1}-y_k}\geq&\,\tfrac{1}{2}\normsq{x_{k+1} - y_k +\lambda_k v_{k+1}}+ \lambda_k(1+\lambda_k\mu)\Big(g(x_{k+1})-g(w_{k+1})\\
        &\,+\tfrac{\mu}{2}\normsq{x_{k+1}-w_{k+1}}
    - \langle x_{k+1}-w_{k+1},v_{k+1} \rangle \Big).
        \end{aligned}\]
    \end{itemize}
    Substituting $y_k$ and $z_{k+1}$ in the weighted sum, that is
    \begin{equation*}
        \begin{array}{ccc}
             y_k &=& x_k+\tfrac{(A_{k+1}-A_k)(1+\mu A_k)}{A_{k+1}+\mu A_k(2 A_{k+1}-A_k)}(z_k-x_{k})  \\
             z_{k+1} &=& z_{k}+\tfrac{A_{k+1}-A_k}{1+\mu A_{k+1}}\left(\mu(x_{k+1}-z_k) -v_{k+1}\right),
        \end{array}
    \end{equation*}
     the weighted sum is equivalently reformulated as
    \begin{align*}
        &A_{k+1}(g(x_{k+1}) - g(x_\star)) + \tfrac{1+\mu A_{k+1}}{2}\normsq{z_{k+1}-x_\star} \\
        &\leq A_{k}(g(x_{k}) - g(x_\star)) + \tfrac{1+\mu A_{k}}{2}\normsq{z_{k}-x_\star}\\
        &\,\,\,\,\,- \tfrac{A_k(A_{k+1}-A_k)\mu (1 + A_k \mu)}{2(A_{k+1} + A_k(2A_{k+1}-A_k)\mu)}\normsq{x_k-z_k}\\
        &\,\,\,\,\,-\tfrac{\lambda_k}{2}\left(\tfrac{\lambda_k\mu(A_{k+1}+A_k(2A_{k+1}-A_k)\mu)}{1+A_{k+1}\mu}  \right)\normsq{\tfrac{y_k-x_{k+1}}{\lambda_k} - \tfrac{\mu(A_{k+1}^2+A_k^2\sigma_k)+A_{k+1}(1-\sigma_k -2A_k\mu\sigma_k)}{A_{k+1}(1+A_{k+1\mu})(1-\sigma_k^2) +\lambda_k \mu\sigma_k(A_{k+1}+A_k(2A_{k+1}-A_k)\mu)}v_{k+1}}\\
        &\,\,\,\,\,- \tfrac{\lambda_k(1-\sigma_k^2)A_{k+1}}{2\sigma_k}\normsq{\tfrac{y_k-x_{k+1}}{\lambda_k} - \tfrac{\mu(A_{k+1}^2+A_k^2\sigma_k)+A_{k+1}(1-\sigma_k -2A_k\mu\sigma_k)}{A_{k+1}(1+A_{k+1\mu})(1-\sigma_k^2) +\lambda_k \mu\sigma_k(A_{k+1}+A_k(2A_{k+1}-A_k)\mu)}v_{k+1}}\\
        &\,\,\,\,\,+ \tfrac{A_{k+1}((2A_kA_{k+1}-A_{k}^2)(1-\sigma_k^2+\lambda_k^2\mu^2+\lambda_k\mu(2-\sigma_k)) - A_{k+1}^2(1-\sigma_k^2+\lambda_k\mu\sigma_k)+A_{k+1}\lambda_k(2(1-\sigma_k)+\lambda_k\mu))}{2(A_k^2\lambda_k\mu^2\sigma_k-A_{k+1}^2\mu(1-\sigma_k^2)-A_{k+1}(1-\sigma_k^2+\lambda_k\mu\sigma_k(1+2A_k\mu)))}\normsq{v_{k+1}}\\
        &\leq A_{k}(g(x_{k}) - g(x_\star)) + \tfrac{1+\mu A_{k}}{2}\normsq{z_{k}-x_\star}\\
        &\,\,\,\,\,+ \tfrac{A_{k+1}((2A_kA_{k+1}-A_{k}^2)(1-\sigma_k^2+\lambda_k^2\mu^2+\lambda_k\mu(2-\sigma_k)) - A_{k+1}^2(1-\sigma_k^2+\lambda_k\mu\sigma_k)+A_{k+1}\lambda_k(2(1-\sigma_k)+\lambda_k\mu))}{2(A_k^2\lambda_k\mu^2\sigma_k-A_{k+1}^2\mu(1-\sigma_k^2)-A_{k+1}(1-\sigma_k^2+\lambda_k\mu\sigma_k(1+2A_k\mu)))}\normsq{v_{k+1}}\\
        &=A_{k}(g(x_{k}) - g(x_\star)) + \tfrac{1+\mu A_{k}}{2}\normsq{z_{k}-x_\star},
    \end{align*}
    where the inequality in the second to last line comes from the fact that factors in front of squared Euclidean norms are nonpositive, and the last equality from the fact that $A_{k+1}$ is chosen such that it satisfies
    \[(2A_kA_{k+1}-A_{k}^2)(1-\sigma_k^2+\lambda_k^2\mu^2+\lambda_k\mu(2-\sigma_k)) - A_{k+1}^2(1-\sigma_k^2+\lambda_k\mu\sigma_k)+A_{k+1}\lambda_k(2(1-\sigma_k)+\lambda_k\mu) = 0.\]
    
        Note that the intermediary expressions largely simplifies when choosing this $A_{k+1}$, as the last term disappears, and the two other squared Euclidean norms become (up to nonpositive multiplicative factors) $\normsq{x_k-z_k}$ and $\normsq{\tfrac{y_k-x_{k+1}}{\lambda_k} - \tfrac{1-\sigma_k + \lambda_k\mu}{1-\sigma_k^2+\lambda_k\mu}v_{k+1}}$.
        
    For the case $\sigma_k = 0$ (i.e., exact proximal computations) $v_{k+1}\in \partial g(x_{k+1})$ and we proceed as previously by performing the following weighted sum of inequalities: 
    \begin{itemize}
        \item strong convexity between $x_{k+1}$ and $x_\star$ with weight $\nu_1 = A_{k+1}-A_k$
        \[g(x_\star)\geq g(x_{k+1}) + \langle v_{k+1}, x_\star-x_{k+1}\rangle + \tfrac{\mu}{2}\normsq{x_{k+1}-x_\star},\]
        \item strong convexity between $x_{k+1}$ and $x_{k}$ with weight $\nu_2 = A_k$
        \[g(x_k)\geq g(x_{k+1}) + \langle v_{k+1}, x_k-x_{k+1}\rangle + \tfrac{\mu}{2}\normsq{x_{k+1}-x_{k}}.\]
    \end{itemize}
    Substituting $y_k$, $x_{k+1}$ and $z_k$ in the weighted sum, that is 
        \begin{equation*}
        \begin{array}{rcl}
             y_k &=& x_k+\tfrac{(A_{k+1}-A_k)(1+\mu A_k)}{A_{k+1}+\mu A_k(2 A_{k+1}-A_k)}(z_k-x_{k})  \\
             x_{k+1} &=& y_k - \lambda_kv_{k+1}\\
             z_{k+1} &=& z_{k}+\tfrac{A_{k+1}-A_k}{1+\mu A_{k+1}}\left(\mu(x_{k+1}-z_k) -v_{k+1}\right),
        \end{array}
    \end{equation*}
    the weighted sum is equivalently reformulated as
    \begin{align*}
        &A_{k+1}(g(x_{k+1}) - g(x_\star)) + \tfrac{1+\mu A_{k+1}}{2}\normsq{z_{k+1}-x_\star} \\
        &\leq A_{k}(g(x_{k}) - g(x_\star)) + \tfrac{1+\mu A_{k}}{2}\normsq{z_{k}-x_\star}+ \tfrac{A_{k+1}(A_{k+1}-\lambda_k(2+\lambda_k\mu)) - A_k(2A_{k+1}-A_k)(1+\lambda_k\mu)^2}{2(1+A_{k+1}\mu)}\normsq{v_{k+1}}\\
        &\,\,\,\,\,- \tfrac{\mu A_k(A_{k+1}-A_k) (1 + A_k \mu)}{2(A_{k+1} + A_k(2A_{k+1}-A_k)\mu)}\normsq{x_k-z_k}\\
        &\leq A_{k}(g(x_{k}) - g(x_\star)) + \tfrac{1+\mu A_{k}}{2}\normsq{z_{k}-x_\star}+ \tfrac{A_{k+1}(A_{k+1}-\lambda_k(2+\lambda_k\mu)) - A_k(2A_{k+1}-A_k)(1+\lambda_k\mu)^2}{2(1+A_{k+1}\mu)}\normsq{v_{k+1}}\\
        & = A_{k}(g(x_{k}) - g(x_\star)) + \tfrac{1+\mu A_{k}}{2}\normsq{z_{k}-x_\star},
    \end{align*}
     where the inequality in the second to last line comes from the fact that factor in front of squared Euclidean norm is nonpositive, and the last equality from the fact that $A_{k+1}$ is chosen such that it satisfies
    \[A_{k+1}(A_{k+1}-\lambda_k(2+\lambda_k\mu)) - A_k(2A_{k+1}-A_k)(1+\lambda_k\mu)^2 = 0,\]
    when $\sigma_k = 0$.
\end{proof}

\section{Conclusion}\label{s:ccl}
In this note, we proposed an inexact accelerated forward-backward method for solving composite convex minimization problems, along with some worst-case guarantees. The method supports inexact evaluations of the proximal subproblems, backtracking line-search on the smoothness parameter, and allows exploiting the possible strong convexity of one of the component in the objective function. The analysis relies on a now standard Lyapunov argument of the same type as that of~\citet{Nesterov:1983wy}, and the numerical behavior is illustrated on a factorization and a total variation problem.

We further provide a version of the hybrid proximal extragradient method~\citep{monteiro2013accelerated} allowing to exploit strong convexity of the objective.

\subsection*{Acknowledgments}
The authors thank Alexander Gasnikov as well as an anonymous referee and an associate editor of OJMO for their careful feedbacks.

\bibliographystyle{plain+eid}
\bibliography{bib}

\end{document}